\documentclass[11pt,a4paper]{article}

\usepackage{graphicx}
\usepackage{multirow}
\usepackage{amsmath,amssymb,amsfonts}
\usepackage{amsthm}
\usepackage{mathrsfs}
\usepackage[title]{appendix}
\usepackage{xcolor}
\usepackage{textcomp}
\usepackage{booktabs}
\usepackage{algorithm}
\usepackage{hyperref}
\usepackage{algpseudocode}
\usepackage{listings}
\usepackage{enumitem}
\usepackage[normalem]{ulem}
\usepackage{authblk} 

\theoremstyle{plain} 
\newtheorem{theorem}{Theorem}[section] 
\newtheorem{lemma}[theorem]{Lemma}
\newtheorem{corollary}[theorem]{Corollary}
\newtheorem{example}[theorem]{Example}
\newtheorem{remark}[theorem]{Remark}
\newtheorem{proposition}[theorem]{Proposition}

\theoremstyle{definition} 
\newtheorem{definition}[theorem]{Definition}

\numberwithin{equation}{section} 

\begin{document}
	
	\title{The best Musielak-Orlicz approximation by linear subspaces}
	
	\author[1]{Juan Costa Ponce}
	\author[2]{Sergio Favier}
	\author[3]{Fabi\'an Levis}
	
	\affil[1]{Instituto de Investigaciones Matem\'aticas Luis A. Santal\'o (IMAS), CONICET--UBA, Buenos Aires, Argentina}
	\affil[2]{Instituto de Matem\'atica Aplicada San Luis Dr. Ezio Marchi (IMASL), CONICET--UNSL, San Luis, Argentina}
	\affil[3]{Departamento de Matem\'atica, FCEFQyN, Universidad Nacional de R\'io Cuarto -- CONICET, R\'io Cuarto, Argentina}
	
	\affil[ ]{\vspace{0.5em} \texttt{\small costaponcejuan@gmail.com, sfavier@unsl.edu.ar, flevis@exa.unrc.edu.ar}}
	
	\date{\today} 
	
	\maketitle
	
	\begin{abstract}
		We study best modular approximation in Musielak--Orlicz spaces. A geometric classification near the origin of Orlicz functions is introduced to isolate weighted $L^1$-like behavior. We establish existence results under minimal hypotheses and provide a variational characterization of minimizers via lateral derivatives. Finally, we establish uniqueness criteria, including a one-dimensional Haar-type theorem and characterizations in $0$- and $1$-spaces in higher dimensions.
	\end{abstract}
	
	\vspace{1em}
	\noindent\textbf{Keywords:} Musielak-Orlicz Spaces, Best Approximation, Convex Functions
	
	\vspace{0.5em}
	\noindent\textbf{MSC Classification:} 46E30, 41A50, 41A52, 41A65
	
	\section{Introduction}
	Starting from the classical Lebesgue spaces $L^p$, two fundamental generalizations are Orlicz spaces $L^\Psi$ and variable exponent Lebesgue spaces $L^{p(\cdot)}$ (see, e.g., \cite{HarjulehtoHasto2019,MendezLang2019}). Both belong to the broader class of Musielak--Orlicz spaces, introduced by Musielak \cite{MusielakOrlicz1959,Musielak1983_monograph} as an extension of the classical theory of Orlicz spaces (see \cite{Krasnoselskii1961}). Their functional-analytic structure is naturally described within the theory of modular spaces \cite{Kozlowski1988}. By allowing the modular $\Phi(x,t)$ to depend on the spatial variable $x$, Musielak--Orlicz spaces provide a flexible framework for modeling nonstandard growth and heterogeneous regularity. Consequently, they have become an important tool in nonlinear analysis, the study of partial differential equations with nonstandard growth, and the calculus of variations.
	
	Best approximation problems in classical $L^p$ spaces are well understood thanks to uniform convexity and reflexivity, which provide powerful variational tools, and in Orlicz spaces were extensively treated in \cite{Landers1980}. In the Musielak--Orlicz setting these global properties may fail: the spatial dependence of $\Phi$ can produce local piecewise affine or staircase behavior near the origin and may destroy the differentiability and convexity assumptions underlying many classical arguments, so the usual variational tools are no longer directly applicable. Understanding existence, variational characterization, and uniqueness of best modular approximants under minimal regularity on $\Phi$ is therefore both mathematically natural and a challenging problem with potential applications whenever nonuniform growth models arise.
	
	The purpose of this paper is to develop a general theory of best modular approximation in Musielak--Orlicz spaces under minimal regularity assumptions on the modular. Our approach neither requires differentiability of $\Phi$ with respect to its second variable nor excludes piecewise affine behavior near the origin. This level of generality allows us to cover important classes of Musielak--Orlicz spaces that fall outside the scope of the existing theory.
	
	Our main contributions are as follows.  We introduce a geometric classification of Orlicz functions near the origin (and extend it to the Musielak--Orlicz setting) to identify regions where the modular behaves like a weighted $L^1$ norm. We then establish an existence theorem for best modular approximants under mild measurability and lower semicontinuity hypotheses, with a corollary for functionals satisfying the $\Delta_2$ condition. Next, we derive a variational characterization of minimizers via lateral partial derivatives of $\Phi$ and a sign-coherence property for distinct minimizers. Finally, we obtain uniqueness criteria, including results based on strict convexity at the origin, a one-dimensional generalization of Haar's theorem, and results for $0$- and $1$-spaces including a necessary and sufficient condition for uniqueness.
	
	In the modular setting Kilmer, Koz{\l}owski, and Lewicki established foundational existence and characterization results \cite{Kilmer1990}. For classical Orlicz spaces, Benavente, Favier and Levis obtained explicit variational characterizations using lateral derivatives \cite{BenaventeFavierLevis2017}, and in \cite{BenaventeCostaPonceFavier2026} the present authors addressed uniqueness when strict convexity is absent by introducing the notions of $0$- and $1$-spaces. Kopaliani, Samashvili and Zviadadze \cite{Kopaliani2021} gave characterization and uniqueness theorems for polynomial approximation under a global differentiability assumption in the second variable; however, this excludes modulars with piecewise affine or staircase behavior near the origin, which we treat here without assuming differentiability of $\Phi$ with respect to its second variable.
	
	The remainder of the paper is organized as follows. Section~\ref{sec:prelim} collects notation and preliminary definitions. Section~\ref{sec:classification} introduces a novel geometric classification of Orlicz functions and applies it to partition the domain of Musielak--Orlicz functionals. Section~\ref{sec:existence} contains the existence theorems. Section~\ref{sec:characterization} provides the variational characterization of best approximants together with their sign-coherence property. Finally, Section~\ref{sec:uniqueness} treats uniqueness: strict convexity at the origin, a one-dimensional Haar generalization, and results for $0$- and $1$-spaces, including a necessary and sufficient condition.

	\section{Definitions and Notation} \label{sec:prelim}
	
	Let $\Omega \subset \mathbb{R}^d$ be a compact domain, $\mathcal{A}$ the $\sigma$-algebra of Lebesgue-measurable subsets of $\Omega$, and $\mu$ the Lebesgue measure. These assumptions will be maintained throughout. 
	
	\begin{definition}\label{orlicz function}
		A function $\Psi: [0, \infty) \to [0, \infty)$ is an \textbf{Orlicz function} if it is convex, with $\Psi(0) = 0$, $\Psi(t) > 0$ for $t > 0$, and $\lim_{t \to \infty} \Psi(t) = \infty$.
		
		For an Orlicz function $\Psi$, its right-hand derivative $\psi^{+}: [0, \infty) \to [0, \infty)$ is defined by $\psi^+(t) := \lim_{h \to 0^+} \frac{\Psi(t+h) - \Psi(t)}{h}$ for $t\ge 0$. 
	\end{definition}
	
	Observe that $\psi^{+}$ is non-decreasing, right-continuous and $\Psi(t) = \int_0^t \psi^+(s) ds$. 
	
	\begin{definition}\label{def:MO_function}
		A function $\Phi: \Omega \times [0, \infty) \to [0, \infty)$ is a \textbf{Musielak-Orlicz function} if the following two conditions hold 
		\begin{enumerate}
			\item[(i)] For each $t\ge0$, the map $x \mapsto \Phi(x, t)$ is $\mathcal{A}$-measurable and $\int_\Omega \Phi(x, t) d\mu(x) < \infty.$
			\item[(ii)] For a.e. $x \in \Omega$, the map $t \mapsto \Phi(x, t)$ is an Orlicz function.
		\end{enumerate}
	\end{definition}
	
	\begin{definition}[see \cite{MendezLang2019}, p. 110]\label{def:delta2}
		A function $\Phi: \Omega \times [0, \infty) \to [0, \infty)$ satisfies the \textbf{$\Delta_2$ condition} if there exist constants $C > 1$ and $T_0 > 0$ such that,
		$$\Phi(x, 2t) \le C\Phi(x, t) \quad \text{for all } t \ge T_0 \text{ and a.e. } x \in \Omega.$$
	\end{definition}
	
	\begin{definition}\label{def:MO_space}
		Let $\Phi$ be a Musielak-Orlicz function.
		Thus for any $\mathcal{A}$-measurable function $f$, we set $$\rho_\Phi(f) = \int_\Omega \Phi(x, |f(x)|) d\mu(x),$$ 
		for the  modular associated with $\Phi.$ 
		
		The \textbf{Musielak-Orlicz space} $L^\Phi(\Omega)$ is the linear space of all $\mathcal{A}$-measurable functions $f: \Omega \to \mathbb{R}$ such that	$$L^\Phi(\Omega) = \left\{ f : \Omega \to \mathbb{R} \text{ $\mathcal{A}$-measurable } \mid \rho_\Phi(\alpha f) < \infty, {\rm for\,\, some}\,\, \alpha > 0\right\}.$$
	\end{definition}
	
	Note that in the definition of $L^{\Phi} (\Omega)$ it can be consider $\alpha = 1$ if the Musielak-Orlicz function $\Phi$  satisfies the $\Delta_2-$condition. 
	
	Additionally, we represent by $C(\Omega)$ the set of all the real-valued continuous functions defined on $\Omega.$ 
	
	\begin{definition}
		Let $S \subset L^\infty(\Omega)$ be a finite-dimensional linear subspace. We refer to $S$ as the \textbf{approximating class}. 
	\end{definition}
	
	\begin{definition}\label{def:best_approximant}
		Let $f \in L^\Phi(\Omega)$. An element $P \in S$ is a \textbf{best modular approximant} to $f$ from $S$ if
		$$\rho_\Phi(f - P) = \inf_{Q \in S} \rho_\Phi(f - Q).$$
		The set of best modular approximants to $f$ from $S$ is denoted by $\mathcal{M}_\Phi(f/S)$.
	\end{definition}
	
	We denote the right and left partial derivatives of $\Phi$ with respect to the second variable as
	$$\varphi_+(x, t) = \frac{\partial^+ \Phi}{\partial t}(x, t) \quad \text{and} \quad \varphi_-(x, t) = \frac{\partial^- \Phi}{\partial t}(x, t),$$
	defined for $t>0$. 
	
	We denote the right-hand limit of $\varphi_+$ at the origin by $\varphi_0(x) = \lim_{t \to 0^+} \varphi_+(x, t)$.
	
	\section{Classification of Orlicz functions} \label{sec:classification}
	
	Following the geometric discussion first introduced in \cite[p. 46]{CostaPonce2024}, we classify Orlicz functions  (Definition \ref{orlicz function}) into three categories according to their behavior near the origin.

	\begin{definition}[Geometric Categories for Orlicz functions]\label{def:categories}
		Let $\Psi$ be an Orlicz function.
		\begin{enumerate}[label=(\roman*)]
			\item If the set $\{a > 0: \psi^{+}\text{ is strictly increasing on } [0, a]\}$ is non-empty, then $\Psi$ is \textbf{initially strict}.
			\item If the set $\{b > 0: \psi^{+}\text{ is constant on } [0, b]\}$ is non-empty, then $\Psi$ is \textbf{initially linear}.
			\item $\Psi$ is a \textbf{vanishing staircase} if there exists a strictly decreasing sequence of positive real numbers $\{c_n\}_{n \in \mathbb{N}}$ such that $\lim_{n \to \infty} c_n = \psi^{+}(0)$, for each $n$ the preimage $(\psi^{+})^{-1}(c_n)$ is a non-degenerate interval $I_n$ (i.e., an interval of positive length), and $\lim_{n \to \infty} \sup I_n = 0$.
		\end{enumerate}
	\end{definition}
	
	The following theorem formalizes the fact that these categories are exhaustive and mutually exclusive.
	
	\begin{theorem}\label{thm:clases_orlicz}
		Every Orlicz function \(\Psi\) belongs to exactly one of the three geometric categories defined in Definition \ref{def:categories}.
	\end{theorem}
	
	\begin{proof}
		It is clear that categories (i) and (ii) are mutually exclusive. 
		
		Suppose that $\Psi$ is neither initially strict nor initially linear. Therefore, on any interval $[0, \epsilon]$ with $\epsilon > 0$, $\psi^+$ is neither strictly increasing nor constant.
		
		We will show constructively that $\Psi$ is a vanishing staircase. 
		
		Let $\epsilon_1 = 1$. For $n \ge 1$:\begin{itemize}
			\item Since $\psi^+$ is non-decreasing, there must be a constant $c_n$ such that the preimage $ (\psi^+)^{-1}(c_n)$ is a non-degenerate interval, which we name $I_n$. Because $\Psi$ is not initially linear, $\inf I_n > 0$. Furthermore, $I_n$ is a proper subset of $(0, \epsilon_n]$.
			\item We set $\epsilon_{n+1} := \inf I_n$.
		\end{itemize}
		
		This process yields a sequence of non-degenerate intervals $\{I_n\}$ such that $\sup I_{n+1} \le \epsilon_{n+1} = \inf I_n < \sup I_n$. Thus, $c_{n+1} < c_n$. 
		
		The sequence $\{\inf I_n\}$ is strictly decreasing and bounded below by $0$, so it must converge to some limit $L \ge 0$. Suppose $L > 0$. Then $\psi^+$ is strictly increasing on $[0,L]$, thus  $\Psi$ is initially strict, which contradicts our assumption. Hence, $L = 0$. It follows that $\lim_{n \to \infty} \sup I_n = 0$.
		
		Finally, since $\psi^+$ is right-continuous, $c_n = \psi^+(\sup I_n)$. Taking the limit as $n \to \infty$ yields $\lim_{n \to \infty} c_n = \psi^+(0)$.
		
		This confirms that $\Psi$ is a vanishing staircase. Therefore, the three categories are mutually exclusive and exhaustive.
	\end{proof}
	
	\begin{example}[A vanishing staircase Orlicz function]
		For each $n \in \mathbb{N}$, define the interval $I_n = \left[ \frac{1}{n+1}, \frac{1}{n} \right)$ for $n \ge 2$ and $I_1=\left[\tfrac12,1\right]$ . Let $\psi^{+}: [0, \infty) \to [0, \infty)$ be defined as:
		\[
		\psi^{+}(t) = \begin{cases} 
			0 & \text{if } t = 0, \\
			\frac{1}{n} & \text{if } t \in I_n \text{ for some } n \in \mathbb{N}, \\
			t & \text{if } t > 1.
		\end{cases}
		\]
		The function $\psi^{+}$ is non-decreasing and right-continuous. The associated Orlicz function is $\Psi(t) = \int_0^t \psi^{+}(s)  ds$. For any $t \in I_n$, the graph of $\Psi$ consists of a straight line segment with slope $1/n$. Since the sequence of constant values $c_n = 1/n$ strictly decreases to $0$ and the preimages $I_n$ are non-degenerate intervals accumulating at $0$, $\Psi$ is neither strictly convex nor purely linear starting near zero, and exhibits the vanishing staircase behavior described in category (iii).
	\end{example}
	
	Motivated by the phenomena we just described, we extend these insights to the generalized Musielak-Orlicz context. To the best knowledge of the authors, this approach is novel. 
	
	\begin{definition}[Domain Partition]\label{def:spatial_partition}
		Let $\Phi$ be a Musielak-Orlicz function. We partition $\Omega$ into the following disjoint subsets,
		\begin{itemize}
			\item \textbf{$\Omega_I$:} The set of points where $\Phi(x, \cdot)$ is initially strict.
			\[ \Omega_I := \bigcup_{n \in \mathbb{N}}  \bigcap_{\substack{q_1, q_2 \in \mathbb{Q} \\ 0 < q_1 < q_2 < 1/n}} \{ x \in \Omega : \varphi_+(x, q_1) < \varphi_+(x, q_2) \}. \]
			
			\item \textbf{$\Omega_{II}$:} The set of points where $\Phi(x, \cdot)$ is initially linear.
			\[ \Omega_{II} := \bigcup_{n \in \mathbb{N}} \{ x \in \Omega : \varphi_+(x, 1/n) = \varphi_0(x) \}. \]
			
			\item \textbf{$\Omega_{III}$:} The complement set, $\Omega_{III} := \Omega \setminus (\Omega_I \cup \Omega_{II})$. 
		\end{itemize}
	\end{definition}
	
	An interesting result that arises from the previous partition is the existence of a measurable function that acts as threshold of strict convexity, properly characterized in the following result, to which we will refer later on.
	
	\begin{proposition}\label{prop:equiv_caso_I}
		Let $\Phi$ be a Musielak--Orlicz function. The following are equivalent:
		\begin{enumerate}[label=(\roman*)]
			\item $\mu(\Omega \setminus \Omega_I) = 0$.
			\item There exists an $\mathcal{A}$-measurable function $\beta:\Omega\to(0,\infty)$ 
			such that $\Phi(x, \cdot)$ is strictly convex on $(0, \beta(x))$ 
			for a.e. $x \in \Omega$.
		\end{enumerate}
	\end{proposition}
	
	\begin{proof}
		\textit{(i) $\Rightarrow$ (ii).} Suppose $\mu(\Omega \setminus \Omega_I) = 0$.
		For a.e. $x \in \Omega$ we have $x \in \Omega_I$, so there exists 
		$n(x) \in \mathbb{N}$ such that $\varphi_+(x, q_1) < \varphi_+(x, q_2)$ 
		for all $q_1, q_2 \in \mathbb{Q}$ with $0 < q_1 < q_2 < 1/n(x)$. Since 
		$\varphi_+(x, \cdot)$ is right-continuous and non-decreasing, being strictly 
		increasing on the dense set $\mathbb{Q} \cap (0, 1/n(x))$ implies it is 
		strictly increasing on all of $(0, 1/n(x))$, and therefore $\Phi(x, \cdot)$ 
		is strictly convex on $(0, 1/n(x))$. Define
		\begin{equation*}
			\beta(x) = \sup \left\{ \delta > 0 : \Phi(x, \cdot) 
			\text{ is strictly convex on } (0, \delta) \right\},
		\end{equation*}
		so that $\Phi(x,\cdot)$ is strictly convex on $(0,\beta(x))$ and 
		$\beta(x) \geq 1/n(x) > 0$ for a.e. $x\in\Omega$. To verify measurability, 
		note that the same argument shows that for any $\alpha > 0$, $\Phi(x,\cdot)$ 
		is strictly convex on $(0,\alpha)$ if and only if $\varphi_+(x,q_1) < 
		\varphi_+(x,q_2)$ for all $q_1, q_2 \in \mathbb{Q}$ with $0 < q_1 < q_2 
		< \alpha$, so that
		\begin{equation*}
			\{ x \in \Omega : \beta(x) \geq \alpha \} 
			= \bigcap_{\substack{q_1, q_2 \in \mathbb{Q} \\ 0 < q_1 < q_2 < \alpha}} 
			\{ x \in \Omega : \varphi_+(x, q_1) < \varphi_+(x, q_2) \}.
		\end{equation*}
		Since each set $\{ x \in \Omega : \varphi_+(x, q_1) < \varphi_+(x, q_2) \}$ 
		is $\mathcal{A}$-measurable by the measurability of $\varphi_+(x,t)$ in $x$ for each fixed 
		$t$, the set $\{ x \in \Omega : \beta(x) \geq \alpha \}$ is a countable 
		intersection of $\mathcal{A}$-measurable sets, hence it is $\mathcal{A}$-measurable. Since $\alpha > 0$ was 
		arbitrary, $\beta$ is an $\mathcal{A}$-measurable function.
		
		\textit{(ii) $\Rightarrow$ (i).} Suppose there exists an $\mathcal{A}$-measurable function 
		$\beta:\Omega\to(0,\infty)$ such that $\Phi(x,\cdot)$ is strictly convex on 
		$(0,\beta(x))$ for a.e. $x\in\Omega$. Then $\varphi_+(x,\cdot)$ is strictly 
		increasing on $(0,\beta(x))$ for a.e. $x\in\Omega$. For a.e. $x\in\Omega$, 
		choose $n(x)\in\mathbb{N}$ such that $1/n(x) < \beta(x)$. Then 
		$\varphi_+(x,q_1) < \varphi_+(x,q_2)$ for all $q_1,q_2\in\mathbb{Q}$ with 
		$0 < q_1 < q_2 < 1/n(x)$, which means precisely that $x \in \Omega_I$. 
		Hence $x \in \Omega_I$ for a.e. $x \in \Omega$, that is, 
		$\mu(\Omega \setminus \Omega_I) = 0$.
	\end{proof}
	
	\begin{remark}	
		By construction, $x \in \Omega_{III}$ if and only if $\varphi_+(x, \cdot)$ exhibits a vanishing staircase behavior near the origin. That is, for each $x \in \Omega_{III}$, there exists a strictly decreasing sequence of positive real numbers $\{c_n(x)\}_{n \in \mathbb{N}}$ with $\lim_{n \to \infty} c_n(x) = \varphi_0(x)$, such that for each $n$, the preimage $I_n(x) = \{t > 0 : \varphi_+(x, t) = c_n(x)\}$ is a non-degenerate interval, and $\lim_{n \to \infty} \sup I_n(x) = 0$.
	\end{remark}
	
	\begin{remark}
		The sets $\Omega_I$, $\Omega_{II}$, and $\Omega_{III}$ are $\mathcal{A}$-measurable, as they are formed by countable unions and intersections of $\mathcal{A}$-measurable sets.
	\end{remark}
	
	\section{Existence} \label{sec:existence}
	
	Regarding the existence of best modular approximants from finite-dimensional linear subspaces, we further develop the results in \cite{BenaventeFavierLevis2017}. The following theorem imposes fairly weak conditions that guarantee the existence of a best approximation.
	
	\begin{theorem}\label{thm:existencia_general}
		Let \(S\subset L^\infty(\Omega)\) be a finite-dimensional linear subspace and let
		\(\Phi:\Omega\times[0,\infty)\to[0,\infty)\) that satisfies the following conditions:
		\begin{enumerate}[label=(\roman*)]
			\item for every \(t\ge0\), the map \(x\mapsto\Phi(x,t)\) is \(\mathcal A\)-measurable;
			\item for almost every \(x\in\Omega\), the map \(t\mapsto\Phi(x,t)\) is nondecreasing, left-continuous, \(\Phi(x,0)=0\) and  \(\lim_{t\to\infty}\Phi(x,t)=\infty\).
		\end{enumerate}
		Let \(f:\Omega\to\mathbb R\) be \(\mathcal A\)-measurable with \(\int_\Omega\Phi(x,|f(x)|)d\mu(x)<\infty\). Then there exists \(P\in S\) such that
		\[
		\int_\Omega\Phi(x,|f(x)-P(x)|)d\mu(x)=\inf_{Q\in S}\int_\Omega\Phi(x,|f(x)-Q(x)|)d\mu(x).
		\]
	\end{theorem}
	
	\begin{proof}
		We first observe that the zero function $0 \in S$, since $S$ is a linear subspace. Thus, $$\inf_{Q \in S} \int_\Omega \Phi(x, \vert{}f(x) - Q(x)\vert{})  d\mu(x) \le \int_\Omega \Phi(x, \vert{}f(x) - 0\vert{})  d\mu(x) = \int_\Omega \Phi(x, \vert{}f(x)\vert{})  d\mu(x) < \infty.$$
		
		Let $\{Q_k\} \subset S$ be a sequence such that $$\int_\Omega \Phi(x, |f(x) - Q_k(x)|)  d\mu(x) \le \inf_{Q \in S} \int_\Omega \Phi(x, |f(x) - Q(x)|)  d\mu(x) + 1/k.$$ Let us set the positive constant $K:=\inf_{Q \in S} \int_\Omega \Phi(x, |f(x) - Q(x)|)  d\mu(x) + 1$. Then $\int_\Omega \Phi(x, |f(x) - Q_k(x)|)  d\mu(x) \le K$ for all $k \in \mathbb{N}$.
		
		We claim that the sequence $\{\|Q_k\|_{L^\infty(\Omega)}\}$ is bounded. Suppose, for a contradiction, that there exists a subsequence, still denoted by $\{Q_k\}$, such that $\|Q_k\|_{L^\infty(\Omega)} \to \infty$ and $Q_k \neq 0$ for all $k$.
		Since the unit sphere $\{Q \in S : \|Q\|_{L^\infty(\Omega)} = 1\}$ is compact in the finite-dimensional subspace $S$, there exists $R \in S$ with $\|R\|_{L^\infty(\Omega)} = 1$ such that $Q_k / \|Q_k\|_{L^\infty(\Omega)}$ converges uniformly to $R$ a.e. in $\Omega$.
		
		Define the sets $$A_k(\lambda) := \{x \in \Omega : |f(x) - Q_k(x)| > \lambda \|Q_k\|_{L^\infty(\Omega)}\}.$$
		We claim that there exist $\lambda > 0$, $\alpha > 0$, and a subsequence (which we continue to denote by $\{Q_k\}$) such that $\mu(A_k(\lambda)) > \alpha$ for all $k$. 
		
		Indeed, if this were false, then for all $\lambda > 0$, $\mu(A_k(\lambda)) \to 0$, implying that the sequence $|f| / \|Q_k\|_{L^\infty(\Omega)} - |Q_k| / \|Q_k\|_{L^\infty(\Omega)}$ converges to $0$ in measure.
		Extracting a subsequence such that $\vert{}f\vert{} / \Vert{}Q_k\Vert{}_{L^\infty(\Omega)} - \vert{}Q_k\vert{} / \Vert{}Q_k\Vert{}_{L^\infty(\Omega)} \to 0$ a.e. in $\Omega$, and observing that $f / \Vert{}Q_k\Vert{}_{L^\infty(\Omega)} \to 0$ a.e. since $f$ is finite and $\Vert{}Q_k\Vert{}_{L^\infty(\Omega)} \to \infty$, which forces $Q_k / \Vert{}Q_k\Vert{}_{L^\infty(\Omega)} \to 0$ a.e. in $\Omega$. By the uniqueness of the limit, $R=0$ a.e. in $\Omega$, which contradicts the fact that $\Vert{}R\Vert{}_{L^\infty(\Omega)} = 1$.
		
		Since $\lim_{t \to \infty} \Phi(x, t) = \infty$ for almost every $x \in \Omega$, by virtue of Egoroff's theorem (see \cite[p. 62]{Folland1999}), for any $\epsilon > 0$ there exists an $\mathcal{A}$-measurable set $E \subset \Omega$ with $\mu(\Omega \setminus E) < \epsilon$ such that $\lim_{t \to \infty} \Phi(x, t) = \infty$ uniformly on $E$.
		We choose $\epsilon = \alpha / 2$. Thus, for any arbitrarily large $M > 0$, there exists $T > 0$ such that $\Phi(x, t) \ge M$ for all $x \in E$ and all $t \ge T$.
		For sufficiently large $k$, we have $\lambda \|Q_k\|_{L^\infty(\Omega)} \ge T$. Consider the intersection $A_k(\lambda) \cap E$.
		Its measure satisfies $$\mu(A_k(\lambda) \cap E) \ge \mu(A_k(\lambda)) - \mu(\Omega \setminus E) > \alpha / 2.$$
		Then we obtain
		\[ \int_\Omega \Phi(x, |f(x) - Q_k(x)|)  d\mu(x) \ge \int_{A_k(\lambda) \cap E} \Phi(x, \lambda \|Q_k\|_{L^\infty(\Omega)})  d\mu(x) \ge M \frac{\alpha}{2}. \]
		Since $M$ can be chosen arbitrarily large, $\int_\Omega \Phi(x, |f(x) - Q_k(x)|)  d\mu(x) \to \infty$ as $k \to \infty$.
		This contradicts the uniform bound $K$.
		Consequently, the sequence $\{Q_k\}$ must be bounded in the $L^\infty(\Omega)$ norm, so it admits a subsequence that we also name $\{Q_k\}$ which converges uniformly a.e. in $\Omega$ to an element $P \in S$.
		This implies point-wise convergence almost everywhere on $\Omega$. Applying Fatou's Lemma and utilizing the lower semi-continuity of $\Phi(x, \cdot)$ (which follows from being non-decreasing and left-continuous), we conclude
		\begin{align*}
			\int_\Omega \Phi(x, |f(x) - P(x)|)  d\mu(x) &\leq \int_\Omega \liminf_{k \to \infty} \Phi(x, |f(x) - Q_k(x)|)  d\mu(x)  \\
			&\leq \liminf_{k \to \infty} \int_\Omega \Phi(x, |f(x) - Q_k(x)|)  d\mu(x) \\ &\leq \inf_{Q \in S} \int_\Omega \Phi(x, |f(x) - Q(x)|)  d\mu(x) \\
		\end{align*}
		This confirms that $ \int_\Omega \Phi(x, |f(x) - P(x)|)  d\mu(x) = \inf_{Q \in S} \int_\Omega \Phi(x, |f(x) - Q(x)|)  d\mu(x).$
	\end{proof}
	
	The next Corollary is a direct consequence of the previous theorem. 
	
	\begin{corollary}\label{cor:existence}
		Let \(\Phi\) be a Musielak--Orlicz function  and let \(S\subset L^\infty(\Omega)\) be a finite-dimensional linear subspace. Then for every \(f\in L^\Phi(\Omega)\) the set \(\mathcal M_\Phi(f/S)\) is nonempty.
	\end{corollary}
	
	\begin{proof} 
		First suppose there exists $P\in S$ such that $\int_\Omega \Phi(x, |f(x) - P(x)|)  d\mu(x) < \infty.$ Then $f-P$ satisfies the hypothesis of Theorem \ref{thm:existencia_general}, thus \(\mathcal M_\Phi(f-P/S)\)\ is nonempty, and then there exists $Q\in S$ such that
		$$\int_\Omega \Phi(x, |f(x) - P(x)-Q(x)|)  d\mu(x) \le \int_\Omega \Phi(x, |f(x) - P(x)-T(x)|)  d\mu(x),$$
		for all $T\in S.$ Thus $P+Q \in \mathcal M_\Phi(f/S).$
		
		On the other hand, if $$\int_\Omega \Phi(x, |f(x) - P(x)|)  d\mu(x) = \infty, $$
		for all $P\in S,$ we have $\mathcal M_\Phi(f-P/S) =S,$  
		and the proof is completed.
		
	\end{proof} 
	
	\section{Characterization} \label{sec:characterization}
	
	At this stage, our goal is to generalize the characterization results from \cite{BenaventeCostaPonceFavier2026}.

	\begin{theorem}[Characterization of the best approximant]\label{thm:characterization}
		Let \(\Phi\) be a Musielak--Orlicz function satisfying the \(\Delta_2\) condition and let \(S\subset L^\infty(\Omega)\) be a finite-dimensional linear subspace. For \(f\in L^\Phi(\Omega)\) and \(P\in S\), \(P\in\mathcal M_\Phi(f/S)\) if and only if for every \(Q\in S\) one has
		\begin{equation}\label{eq:ineq_characterization}
			\begin{split}
				&\quad \int_{\{(f-P)Q>0\}} \varphi_-(x,|f(x)-P(x)|)\operatorname{sgn}(f(x)-P(x))Q(x)d\mu(x) \\
				& +\int_{\{(f-P)Q<0\}} \varphi_+(x,|f(x)-P(x)|)\operatorname{sgn}(f(x)-P(x))Q(x)d\mu(x) \\
				&\le \int_{\{f=P\}} \varphi_+(x,0)|Q(x)|d\mu(x),
			\end{split}
		\end{equation}
		where \(f-P\) denotes the function \(x\mapsto f(x)-P(x)\) and \(\varphi_\pm\) are the appropriate lateral derivatives associated to \(\Phi\).
	\end{theorem}
	
	\begin{proof}
		For $P \in \mathcal{M}_\Phi(f/S)$ and $Q \in S$, we define the functional $$F_Q(\epsilon) := \int_\Omega \Phi(x, |f(x) - (P(x) + \epsilon Q(x))|)  d\mu(x)$$ for $\epsilon \ge 0$.
		By the convexity of $\Phi(x, \cdot)$, $F_Q$ is a convex function on $[0, \infty)$. Indeed, for any $\epsilon_1, \epsilon_2 \ge 0$ and any $\alpha, \beta \ge 0$ such that $\alpha + \beta = 1$, we can rewrite,
		\[
		f(x) - (P(x) + (\alpha \epsilon_1 + \beta \epsilon_2)Q(x)) = \alpha(f(x) - P(x) - \epsilon_1 Q(x)) + \beta(f(x) - P(x) - \epsilon_2 Q(x)).
		\]
		Thus we have,
		\[
		|f(x) - (P(x) + (\alpha \epsilon_1 + \beta \epsilon_2)Q(x))| \le \alpha|f(x) - P(x) - \epsilon_1 Q(x)| + \beta|f(x) - P(x) - \epsilon_2 Q(x)|.
		\]
		Since $\Phi(x, \cdot)$ is non-decreasing and convex for a. e. $x \in \Omega$, it follows that:
		\begin{align*}
			& \Phi(x, |f(x) - (P(x) + (\alpha \epsilon_1 + \beta \epsilon_2)Q(x))|) \\
			&\le \Phi(x, \alpha|f(x) - P(x) - \epsilon_1 Q(x)| + \beta|f(x) - P(x) - \epsilon_2 Q(x)|) \\
			&\le \alpha \Phi(x, |f(x) - P(x) - \epsilon_1 Q(x)|) + \beta \Phi(x, |f(x) - P(x) - \epsilon_2 Q(x)|).
		\end{align*}
		Integrating this inequality over $\Omega$ with respect to $\mu$ yields the convexity of the functional:
		\[
		F_Q(\alpha \epsilon_1 + \beta \epsilon_2) \le \alpha F_Q(\epsilon_1) + \beta F_Q(\epsilon_2).
		\]
		
		Since $P\in\mathcal{M}_\Phi(f/S)$, $F_Q$ attains its minimum at $\epsilon = 0$, which implies $F_Q^+(0) \ge 0$. We compute this derivative as
		\begin{equation*}
			F_Q^+(0) = \lim_{\epsilon \to 0^+} \int_{\{x\in\Omega:Q(x) \ne 0\}} \frac{\Phi(x, |f(x) - (P(x) + \epsilon Q(x))|) - \Phi(x, |f(x) - P(x)|)}{\epsilon Q(x)} Q(x)  d\mu(x).
		\end{equation*}
		
		To justify the application of Lebesgue's Dominated Convergence Theorem, we must establish an integrable upper bound for
		\[
		\left| \frac{\Phi(x, |f(x) - P(x) - \epsilon Q(x)|) - \Phi(x, |f(x) - P(x)|)}{\epsilon} \right|.
		\]
		We observe that
		\[
		-\epsilon|Q(x)| \le |f(x) - P(x) - \epsilon Q(x)| - |f(x) - P(x)| \le \epsilon|Q(x)|.
		\]
		Because $\Phi(x, \cdot)$ is convex, for any $u, v \ge 0$, $\Phi(x, v) - \Phi(x, u) \le \varphi_+(x, v)(v - u)$ when $v \ge u$, and $\Phi(x, u) - \Phi(x, v) \le \varphi_+(x, u)(u - v)$ when $u > v$.
		In particular, we can set $u = |f(x) - P(x)|$ and $v = |f(x) - P(x) - \epsilon Q(x)|$. Therefore:
		\begin{align*}
			&|\Phi(x, |f(x) - P(x) - \epsilon Q(x)|) - \Phi(x, |f(x) - P(x)|)| \\
			& \le \epsilon|Q(x)| \max\{ \varphi_+(x, |f(x) - P(x) - \epsilon Q(x)|), \varphi_+(x, |f(x) - P(x)|) \}.
		\end{align*}
		Since $\varphi_+(x, \cdot)$ is non-decreasing and we consider limits as $\epsilon \to 0^+$, we can assume $0 < \epsilon \le 1$. Thus, $|f(x) - P(x) - \epsilon Q(x)| \le |f(x) - P(x)| + |Q(x)|$. Therefore,
		\begin{equation*}
			\textstyle
			\left| \frac{\Phi(x, |f(x) - P(x) - \epsilon Q(x)|) - \Phi(x, |f(x) - P(x)|)}{\epsilon} \right| \le |Q(x)| \varphi_+(x, |f(x) - P(x)| + |Q(x)|).
		\end{equation*}
		Since $Q \in L^\infty(\Omega)$, there exists a constant $M_Q > 0$ such that $|Q(x)| \le M_Q$ a. e. in $\Omega$.
		Let us define $g(x) := |f(x) - P(x)| + |Q(x)|$.
		Thus, it suffices to prove the integrability of $\varphi_+(x, g(x))$.
		
		By the convexity of $\Phi(x, \cdot)$, $t \varphi_+(x, t) \le \Phi(x, 2t)$. On the set $\{x \in \Omega : g(x) \ge 1\}$:
		\[
		\varphi_+(x, g(x)) \le g(x)\varphi_+(x, g(x)) \le \Phi(x, 2g(x)),
		\]
		which is integrable because the $\Delta_2$ condition guarantees that $L^\Phi(\Omega)$ is a linear space. 
		
		On the set $\{x \in \Omega : g(x) < 1\}$ we have
		\[
		\varphi_+(x, g(x)) \le \varphi_+(x, 1) \le \Phi(x, 2).
		\]
		The bound is integrable by definition of Musielak-Orlicz function.
		
		Therefore, $|Q(x)| \varphi_+(x, g(x))$ is integrable on $\Omega$, and Lebesgue's Dominated Convergence Theorem is applicable.
		
		We evaluate the pointwise limit of the quotient as $\epsilon \to 0^+$ and we partition the set $\{x \in \Omega : Q(x) \ne 0\}$ into the following disjoint subsets
		\begin{align*}
			A_1 &= \{x \in \Omega : Q(x) > 0\} \cap \{x \in \Omega : P(x) < P(x) + \epsilon Q(x) < f(x)\} \\
			A_2 &= \{x \in \Omega : Q(x) < 0\} \cap \{x \in \Omega : P(x) + \epsilon Q(x) < P(x) < f(x)\} \\
			A_3 &= \{x \in \Omega : Q(x) < 0\} \cap \{x \in \Omega : P(x) + \epsilon Q(x) < f(x) < P(x)\} \\
			A_4 &= \{x \in \Omega : Q(x) < 0\} \cap \{x \in \Omega : P(x) + \epsilon Q(x) < f(x) = P(x)\} \\
			A_5 &= \{x \in \Omega : Q(x) > 0\} \cap \{x \in \Omega : f(x) \le P(x) < P(x) + \epsilon Q(x)\} \\
			A_6 &= \{x \in \Omega : Q(x) > 0\} \cap \{x \in \Omega : P(x) < f(x) \le P(x) + \epsilon Q(x)\} \\
			A_7 &= \{x \in \Omega : Q(x) < 0\} \cap \{x \in \Omega : f(x) \le P(x) + \epsilon Q(x) < P(x)\}.
		\end{align*}
		
		Thus we have,
		\begin{align*} 
			F_Q^+(0) 
			&= \lim_{\epsilon \to 0^+} \sum_{i=1}^7 \int_{A_i} \frac{\Phi(x, \lvert f(x) - (P(x) + \epsilon Q(x)) \rvert) - \Phi(x, \lvert f(x) - P(x) \rvert)}{\epsilon Q(x)} Q(x) \, d\mu(x) \\
			&= \lim_{\epsilon \to 0^+} \sum_{i=1}^7 I_i. 
		\end{align*}
		
		We compute the pointwise limit on each subset.
		\begin{itemize}
			\item On $A_1$, $Q > 0$ and $f - P > \epsilon Q > 0$, so $|f-P-\epsilon Q| = (f-P) - \epsilon Q \nearrow |f-P|$ as $\epsilon \to 0^+$; since the numerator is negative and $\epsilon Q > 0$, the quotient converges to $-\varphi_-(x, |f-P|)$, and multiplication by $Q$ yields $-\varphi_-(x,|f-P|)Q$.
			\item On $A_2$, $Q < 0$ and $f - P > 0$, so $|f-P-\epsilon Q| = (f-P) + \epsilon|Q| \searrow |f-P|$ as $\epsilon \to 0^+$; writing $Q = -|Q|$, the quotient equals $-(\Phi(x,(f-P)+\epsilon|Q|) - \Phi(x,f-P))/(\epsilon|Q|)$, which converges to $-\varphi_+(x, |f-P|)$, and multiplication by $Q$ yields $-\varphi_+(x,|f-P|)Q$.
			\item On $A_4$, $Q < 0$ and $f = P$, so $|f-P-\epsilon Q| = \epsilon|Q|$; writing $Q = -|Q|$, the quotient reduces to $-(\Phi(x,\epsilon|Q|) - \Phi(x,0))/(\epsilon|Q|)$, which converges to $-\varphi_+(x,0)$, and multiplication by $Q = -|Q|$ yields $\varphi_+(x,0)|Q|$.
			\item On $A_5$, $Q > 0$ and $f \le P$, so $|f-P-\epsilon Q| = |f-P| + \epsilon Q \searrow |f-P|$ as $\epsilon \to 0^+$; since both numerator and denominator are positive, the quotient converges to $\varphi_+(x,|f-P|)$, and multiplication by $Q$ yields $\varphi_+(x,|f-P|)Q$.
			\item On $A_7$, $Q < 0$ and $f \le P+\epsilon Q < P$, so $f < P$ in the limit and $|f-P-\epsilon Q| = (P+\epsilon Q-f) \nearrow (P-f) = |f-P|$ as $\epsilon \to 0^+$; writing $Q = -|Q|$, the quotient equals $(\Phi(x,|f-P|-\epsilon|Q|) - \Phi(x,|f-P|))/(-\epsilon|Q|)$, which converges to $\varphi_-(x,|f-P|)$, and multiplication by $Q$ yields $\varphi_-(x,|f-P|)Q$.
		\end{itemize}
		Then we have  
		\begin{align*}
			\lim_{\epsilon \to 0^+} I_1 &= -\int_{\{x \in \Omega : Q(x)>0, f(x)>P(x)\}} \varphi_-(x, |f(x)-P(x)|) Q(x)  d\mu(x), \\
			\lim_{\epsilon \to 0^+} I_2 &= -\int_{\{x \in \Omega : Q(x)<0, f(x)>P(x)\}} \varphi_+(x, |f(x)-P(x)|) Q(x)  d\mu(x), \\
			\lim_{\epsilon \to 0^+} I_4 &= \int_{\{x \in \Omega : Q(x)<0, f(x)=P(x)\}} \varphi_+(x, 0) |Q(x)|  d\mu(x), \\
			\lim_{\epsilon \to 0^+} I_5 &= \int_{\{x \in \Omega : Q(x)>0, f(x)<P(x)\}} \varphi_+(x, |f(x)-P(x)|) Q(x)  d\mu(x) \\
			&\quad + \int_{\{x \in \Omega : Q(x)>0, f(x)=P(x)\}} \varphi_+(x, 0) |Q(x)|  d\mu(x), \\
			\lim_{\epsilon \to 0^+} I_7 &= \int_{\{x \in \Omega : Q(x)<0, f(x)<P(x)\}} \varphi_-(x, |f(x)-P(x)|) Q(x)  d\mu(x).
		\end{align*}
		
		For the remaining integrals, let us consider a decreasing and convergent sequence $\{\epsilon_n\} \searrow 0$, and the decreasing sets $A_{3,n} = \{x \in \Omega : Q(x)<0, P(x)+\epsilon_n Q(x)<f(x)<P(x)\}$. Then $\bigcap_{n=1}^\infty A_{3,n} = \emptyset$, and by the Dominated Convergence Theorem, we have $\lim_{\epsilon \to 0^+} I_3 = 0$. Analogously, let us consider the sets $A_{6,n} = \{x \in \Omega : Q(x)>0, P(x) < f(x) \le P(x)+\epsilon_n Q(x)\}$. Then $\bigcap_{n=1}^\infty A_{6,n} = \emptyset$, and again, by the Dominated Convergence Theorem, we have $\lim_{\epsilon \to 0^+} I_6 = 0$.
		
		Finally, collecting all the above limits, we group the domains where $(f(x)-P(x))Q(x) > 0$ and $(f(x)-P(x))Q(x) < 0$ by factoring the sign, which yields
		\begin{align*}
			F_Q^+(0) &= -\int_{\{(f(x)-P(x))Q(x) > 0\}} \varphi_-(x, |f(x)-P(x)|) \text{sgn}(f(x)-P(x)) Q(x)  d\mu(x) \\
			&\quad - \int_{\{(f(x)-P(x))Q(x) < 0\}} \varphi_+(x, |f(x)-P(x)|) \text{sgn}(f(x)-P(x)) Q(x)  d\mu(x) \\
			&\quad + \int_{\{x \in \Omega : f(x)=P(x)\}} \varphi_+(x, 0) |Q(x)|  d\mu(x).
		\end{align*}
		
		Substituting these limits into the minimality condition $0 \le F_Q^+(0)$ and rearranging the terms yields the stated inequality \ref{eq:ineq_characterization}.
		
		Conversely, assume that the inequality \ref{eq:ineq_characterization} holds for all $Q \in S$. This means that $F_Q^+(0) \ge 0$ for every $Q \in S$. By the convexity of the functional $F_Q$ on $[0, \infty)$, for $\epsilon = 1$ we have
		\[ F_Q(1) - F_Q(0) \ge F_Q^+(0)(1 - 0) \ge 0. \]
		This implies $F_Q(1) \ge F_Q(0)$, which translates to
		\[ \int_\Omega \Phi(x, |f(x) - (P(x) + Q(x))|)  d\mu(x) \ge \int_\Omega \Phi(x, |f(x) - P(x)|)  d\mu(x). \]
		Since $S$ is a linear subspace, any element in $S$ can be written as $P+Q$ for some $Q \in S$. Thus, $P \in \mathcal{M}_\Phi(f/S)$.
	\end{proof}
	
	\begin{lemma}\label{lem:jackson_musielak}
		Under the same hypotheses as in Theorem \ref{thm:characterization}, if \(P_1,P_2\in\mathcal M_\Phi(f/S)\) for some \(f\in L^\Phi(\Omega)\), and defining \(g_1:=f-P_1\) and \(g_2:=f-P_2\), we have
		\[
		g_1(x)g_2(x)\ge0\quad\text{for almost every }x\in\Omega.
		\]
	\end{lemma}
	
	\begin{proof}
		Since $P_1, P_2 \in \mathcal{M}_\Phi(f/S)$, we can set $\rho := \rho_\Phi(g_1) = \rho_\Phi(g_2)$.
		
		Since $S$ is a linear subspace, $\frac{P_1 + P_2}{2} \in S$. By the convexity of $\Phi(x, \cdot)$, for almost every $x \in \Omega$ we obtain:
		\[
		\rho_\Phi\left( f - \frac{P_1 + P_2}{2} \right) \le \frac{1}{2} \rho_\Phi(g_1) + \frac{1}{2} \rho_\Phi(g_2) = \rho.
		\]
		This implies that $\frac{P_1 + P_2}{2} \in \mathcal{M}_\Phi(f/S)$, which forces the modular inequality above to be an equality. Consequently, the pointwise inequality
		\[
		\Phi\left(x, \frac{|g_1(x) + g_2(x)|}{2}\right) \le \frac{1}{2}\Phi(x, |g_1(x)|) + \frac{1}{2}\Phi(x, |g_2(x)|)
		\]
		must hold as an equality for almost every $x \in \Omega$.
		
		Assume, for the sake of contradiction, that there exists an $\mathcal{A}$-measurable set $A \subset \Omega$ with $\mu(A) > 0$ where $g_1(x)$ and $g_2(x)$ have opposite signs. For every $x \in A$, the strict inequality $|g_1(x) + g_2(x)| < |g_1(x)| + |g_2(x)|$ holds.
		
		By the definition of Musielak-Orlicz functions, $\Phi(x, t) = 0$ if and only if $t = 0$. This property, combined with convexity, ensures that $\Phi(x, \cdot)$ is strictly increasing on $[0, \infty)$. Indeed, for any $0 \le u < v$, setting $\lambda = u/v \in [0, 1)$  and using convexity and $\Phi(x,0) = 0$ yields $$\Phi(x,u) = \Phi(x, \lambda v) \le \lambda\Phi(x,v) + (1-\lambda)\Phi(x,0) = \lambda\Phi(x,v) < \Phi(x,v).$$
		Therefore, for $x \in A$, applying this strict monotonicity to the sum of the absolute values gives:
		\[
		\Phi\left(x, \frac{|g_1(x) + g_2(x)|}{2}\right) < \Phi\left(x, \frac{|g_1(x)| + |g_2(x)|}{2}\right) \le \frac{1}{2}\Phi(x, |g_1(x)|) + \frac{1}{2}\Phi(x, |g_2(x)|).
		\]
		
		Since this strict inequality holds pointwise on the set $A$ of positive measure, integrating over the entire domain $\Omega$ we get 
		\begin{align*}
			\rho_\Phi\left( f - \frac{P_1 + P_2}{2} \right) &= \int_\Omega \Phi\left(x, \frac{|g_1(x) + g_2(x)|}{2}\right)  d\mu(x) \\
			&< \int_\Omega \left( \frac{1}{2}\Phi(x, |g_1(x)|) + \frac{1}{2}\Phi(x, |g_2(x)|) \right)  d\mu(x) = \rho.
		\end{align*}
		This strict inequality, $\rho_\Phi\left( f - \frac{P_1 + P_2}{2} \right) < \rho,$ contradicts the fact that $\rho$ is the infimum over $S$. Thus, $\mu(A) = 0$, concluding the proof.
	\end{proof}
	
	\section{Uniqueness} \label{sec:uniqueness}
	
	It is known that global strict convexity of the generating function guarantees uniqueness of the best approximant in Musielak–Orlicz spaces. The theorem below formulates this classical uniqueness result in a localized form: it suffices that the generating function be strictly convex with respect to its variable $t$ on an interval $[0, \beta(x))$, where $\beta$ is the strictly positive measurable function on $\Omega$ characterized in \ref{prop:equiv_caso_I}.
	
	\begin{theorem}\label{thm:unicidad_esszero}
		Let \(\Phi\) be a Musielak--Orlicz function satisfying \(\Delta_2\) and suppose \(\mu(\Omega\setminus\Omega_I)=0\). Let \(S\subset L^\infty(\Omega)\) be a linear subspace and let \(f\in L^\Phi(\Omega)\) be such that for every \(Q\in\mathcal M_\Phi(f/S)\) and every $\mathcal{A}$-measurable set \(A\subset\Omega\) with \(\mu(A)>0\), the following equality holds:
		\[
		\operatorname{ess\,inf}_{x\in A}|f(x)-Q(x)|=0.
		\]
		Then \(\mathcal M_\Phi(f/S)\) is a singleton.
	\end{theorem}
	
	\begin{proof}
		Suppose there exist two distinct best modular approximants 
		$P_1, P_2 \in \mathcal{M}_\Phi(f/S)$, and let
		\[
		\rho = \inf_{Q \in S} \rho_\Phi(f - Q) 
		= \rho_\Phi(f - P_1) = \rho_\Phi(f - P_2).
		\]
		Let $E = \{x \in \Omega : P_1(x) \neq P_2(x)\}$, and assume $\mu(E) > 0$.
		Since $S \subset L^\infty(\Omega)$ is a linear subspace and $\rho_\Phi$ is 
		convex, the average $P^* = \frac{P_1+P_2}{2}$ belongs to $S$ and satisfies
		\[
		\rho_\Phi(f-P^*) \leq \frac{1}{2}\rho_\Phi(f-P_1)
		+\frac{1}{2}\rho_\Phi(f-P_2) = \rho,
		\]
		hence $P^* \in \mathcal{M}_\Phi(f/S)$.
		
		Since $\mu(\Omega \setminus \Omega_I) = 0$, by Proposition~\ref{prop:equiv_caso_I} there 
		exists an $\mathcal{A}$-measurable function $\beta:\Omega\to(0,\infty)$ such that 
		$\Phi(x, \cdot)$ is strictly convex on $(0, \beta(x))$ for a.e. $x \in \Omega$.
		
		Setting $g_i = f - P_i$ for $i=1,2$, Lemma~\ref{lem:jackson_musielak} gives 
		$g_1(x)g_2(x) \geq 0$ for a.e. $x \in \Omega$, that is, $g_1(x)$ and $g_2(x)$ 
		have the same sign or at least one of them vanishes. In either case, for 
		a.e. $x \in E$,
		\begin{equation}\label{eq:average}
			|f(x)-P^*(x)| = \left|\frac{g_1(x)+g_2(x)}{2}\right| 
			= \frac{|g_1(x)|+|g_2(x)|}{2},
		\end{equation}
		and in particular
		\[
		|f(x)-P_i(x)| \leq 2|f(x)-P^*(x)|, \quad i=1,2, 
		\quad \text{a.e. on } E.
		\]
		Moreover, since $P_1(x)\neq P_2(x)$ on $E$, we have $g_1(x)\neq g_2(x)$, 
		and combined with $g_1(x)g_2(x)\geq 0$ this gives $|g_1(x)|\neq|g_2(x)|$ 
		for a.e. $x\in E$: indeed, if both are nonzero they share the same sign and 
		differ in value, while if one vanishes the other does not. Thus the two 
		arguments of $\Phi(x,\cdot)$ appearing in the convexity inequality are 
		distinct for a.e. $x \in E$.
		
		Consider the sets $F_m = \{x \in E : \beta(x) \geq 1/m\}$ for $m \in \mathbb{N}$.
		Since $\beta(x) > 0$ a.e., we have $F_m \nearrow E$ up to a null set, and 
		since $\mu(E) > 0$, there exists $m_0 \in \mathbb{N}$ such that 
		$\mu(F_{m_0}) > 0$.
		
		Applying the hypothesis to $A = F_{m_0}$ and $P^* \in \mathcal{M}_\Phi(f/S)$,
		\[
		\operatorname{ess\,inf}_{x \in F_{m_0}} |f(x) - P^*(x)| = 0,
		\]
		so for every $n \in \mathbb{N}$ the set 
		\[
		A_n = \left\{x \in F_{m_0} : |f(x) - P^*(x)| < \frac{1}{2n}\right\}
		\]
		has positive measure. Fix $n_0 \geq m_0$. For a.e. $x \in A_{n_0}$, using 
		\eqref{eq:average} and the fact that $A_{n_0} \subset F_{m_0}$,
		\[
		|f(x) - P_i(x)| \leq 2|f(x) - P^*(x)| < \frac{1}{n_0} 
		\leq \frac{1}{m_0} \leq \beta(x), \quad i=1,2.
		\]
		Thus, for a.e. $x \in A_{n_0}$, both $|f(x)-P_1(x)|$ and $|f(x)-P_2(x)|$ 
		are distinct values lying in $[0,\beta(x))$, where $\Phi(x,\cdot)$ is strictly 
		convex. Applying strict convexity together with \eqref{eq:average} yields
		\[
		\Phi\!\left(x, |f(x)-P^*(x)|\right) < \frac{1}{2}\Phi\!\left(x, |f(x)-P_1(x)|\right) 
		+ \frac{1}{2}\Phi\!\left(x, |f(x)-P_2(x)|\right)
		\quad \text{a.e. on } A_{n_0}.
		\]
		On $\Omega \setminus A_{n_0}$, the standard convexity of $\Phi(x,\cdot)$ gives 
		the non-strict inequality. Since $\mu(A_{n_0}) > 0$, integrating over $\Omega$ 
		yields
		\[
		\rho_\Phi(f - P^*) < \frac{1}{2}\rho_\Phi(f-P_1) 
		+ \frac{1}{2}\rho_\Phi(f-P_2) = \rho,
		\]
		contradicting $P^* \in \mathcal{M}_\Phi(f/S)$. Hence $\mu(E) = 0$, 
		that is, $P_1 = P_2$ almost everywhere.
	\end{proof}
	
	In order to illustrate the uniqueness result established in the previous Theorem, we present two examples in both of which $f$ is nowhere continuous. 
	
	\begin{example} \label{exa-1}
		Consider $\Omega = [0, 1]$ and let $\Phi: [0, 1] \times [0, \infty) \to [0, \infty)$ be a Musielak-Orlicz function defined by $\Phi(x, t) = \Phi(t)$, where
		\begin{equation*}
			\Phi(t) = \begin{cases}
				t^2 & \text{if } 0 \le t \le 1, \\
				2t - 1 & \text{if } t > 1.
			\end{cases}
		\end{equation*}
		This function satisfies the $\Delta_2$ condition and $\mu(\Omega \setminus \Omega_I) = 0$.
		
		Let $S = \text{span}\left\{\chi_{[0, 1/2)}, \chi_{[1/2, 1]}\right\}$ be the subspace of step functions.  
		
		We construct a function $f$ whose image is dense in the interval $[-M, M]$ for $M > 1$. Let $\{q_n\}_{n \in \mathbb{N}}$ be an enumeration of $\mathbb{Q} \cap [-M, M]$ and let $\{E_n\}_{n \in \mathbb{N}}$ be a partition of $[0, 1]$ into sets that are dense in measure, i.e., $\mu(E_n \cap A) > 0$ for every $\mathcal{A}$-measurable set $A \subset [0, 1]$ with $\mu(A) > 0$. Define
		\begin{equation*}
			f(x) = q_n \quad \text{if } x \in E_n.
		\end{equation*}
		
		Let $Q \in \mathcal{M}_\Phi(f/S)$ be a best modular approximant, say 
		$$Q(x) = c_1 \chi_{[0, 1/2)}(x) + c_2 \chi_{[1/2, 1]}(x).$$ We claim that 
		$c_1, c_2 \in [-M, M]$. Indeed, suppose for contradiction that $c_1 > M$ 
		(the case $c_1 < -M$ is analogous). Define the projected function 
		$$Q^*(x) = M \cdot \chi_{[0,1/2)}(x) + c_2 \chi_{[1/2,1]}(x) \in S.$$
		Since $f(x) \in [-M, M]$ a.e., for almost every $x \in [0, 1/2)$ we have 
		$f(x) \le M < c_1$, which gives:
		\begin{equation*}
			|f(x) - M| = M - f(x) < c_1 - f(x) = |f(x) - c_1|.
		\end{equation*}
		Since $\Phi$ is strictly increasing on $[0,\infty)$, 
		the strict inequality $\Phi(|f(x)-M|) < \Phi(|f(x)-c_1|)$ holds on a set 
		of positive measure. Integrating over $[0,1]$:
		\begin{equation*}
			\rho_\Phi(f - Q^*) < \rho_\Phi(f - Q),
		\end{equation*}
		contradicting the optimality of $Q$. Therefore $c_1 \in [-M,M]$, and by 
		the same argument $c_2 \in [-M,M]$.

		Let $A \subset [0, 1]$ be an $\mathcal{A}$-measurable set with $\mu(A) > 0$. Without loss of generality, assume $\mu(A \cap [0, 1/2)) > 0$. For any $\varepsilon > 0$, since $c_1 \in [-M, M]$, the density of $\mathbb{Q} \cap [-M, M]$ guarantees the existence of a rational $q_{n_0}$ such that $|q_{n_0} - c_1| < \varepsilon$. 
		
		By the construction of $f$, the set where $f(x) = q_{n_0}$ is $E_{n_0}$. Since $E_{n_0}$ is dense in measure, the intersection $E_{n_0} \cap A \cap [0, 1/2)$ has positive measure. On this set, we have
		\begin{equation*}
			|f(x) - Q(x)| = |q_{n_0} - c_1| < \varepsilon.
		\end{equation*}
		Since this holds for every $\varepsilon > 0$, it follows that $\operatorname{ess\,inf}_{x \in A} |f(x) - Q(x)| = 0$, satisfying the theorem's hypothesis. Consequently, the best modular approximant of $f$ in $S$ is unique.
	\end{example}
	
	\begin{example}
		Let $\Phi$ be defined as in the previous example and let 
		$S = \text{span}\{e^x, e^{-x}\}$, which is a two-dimensional subspace 
		of $L^\infty([0, 1])$. 	We construct a function $f$ whose values are prescribed by a dense 
		family of continuous functions. Let $\{s_n\}_{n\in\mathbb{N}}$ be a 
		sequence that is dense in $C([0,1])$ with respect to the supremum norm 
		$\|\cdot\|_\infty$ (for instance, an enumeration of all polynomials with 
		rational coefficients), and let $\{E_n\}_{n\in\mathbb{N}}$ be a partition 
		of $[0,1]$ as in Example~\ref{exa-1}. Define
		\[
		f(x) = s_n(x), \quad x \in E_n.
		\]
		
		Let $Q \in \mathcal{M}_\Phi(f/S)$ be a best modular approximant. 
		Let $A \subset [0, 1]$ be any $\mathcal{A}$-measurable set with $\mu(A) > 0$ and let 
		$\varepsilon > 0$. Since $\{s_n\}_{n \in \mathbb{N}}$ is dense in 
		$C([0,1])$ with respect to $\|\cdot\|_\infty$, there exists $s_{n_0}$ 
		such that:
		\begin{equation*}
			\sup_{x \in [0,1]} |s_{n_0}(x) - Q(x)| < \varepsilon.
		\end{equation*}
		Since $E_{n_0}$ is dense in measure, the intersection $E_{n_0} \cap A$ 
		has positive measure. For any $x \in E_{n_0} \cap A$:
		\begin{equation*}
			|f(x) - Q(x)| = |s_{n_0}(x) - Q(x)| \le 
			\sup_{x \in [0,1]}|s_{n_0}(x) - Q(x)| < \varepsilon.
		\end{equation*}
		Since this holds for every $\varepsilon > 0$, it follows that 
		$\operatorname{ess\,inf}_{x \in A} |f(x) - Q(x)| = 0$, satisfying the 
		theorem's hypothesis. Consequently, the best modular approximant of $f$ 
		in $S$ is unique.
	\end{example}
	
	Haar Spaces, alternatively named Tchebycheff systems in part of the pertaining literature, have been studied for their interpolation properties that resemble the polynomials.
	
	\begin{definition}[Haar Space]
		A subspace $S \subset C([a, b])$ of dimension $n$ is called a Haar space if every non-zero element $P \in S$ has at most $n-1$ distinct roots in $[a, b]$.
	\end{definition}
	
	Below we state a uniqueness result for best modular approximation when the approximating space is a Haar (Tchebycheff) space on a one-dimensional interval. Haar spaces mirror polynomial interpolation and, by the Mairhuber–Curtis theorem, do not exist on multidimensional domains with nonempty interior; hence we restrict to $\Omega=[a,b]$. For examples and further details on Haar spaces see \cite{key:K, key:Sc}.
	
	\begin{theorem}[Uniqueness in Haar spaces]\label{thm:uniqueness_haar}
		Let \(\Omega=[a,b]\) and let \(\Phi\) be a Musielak--Orlicz function satisfying \(\Delta_2\). Let \(S\subset C([a,b])\) be a Haar space of dimension \(n\) and let \(f\in C([a,b])\setminus S\). Then \(\mathcal M_\Phi(f/S)\) is a singleton.
	\end{theorem}
	
	\begin{proof}
		Suppose, for the sake of contradiction, that there exist two distinct best modular approximants
		$P_1, P_2 \in \mathcal{M}_\Phi(f/S)$. By Lemma~\ref{lem:jackson_musielak}, the errors satisfy
		\[
		\left[f(x)-P_1(x)\right]\left[f(x)-P_2(x)\right]\ge 0
		\quad\text{for a.e. }x\in[a,b].
		\tag{1}
		\]
		We distinguish two cases.
		
		\medskip\noindent\textit{Case 1: There exists $P_1\in\mathcal{M}_\Phi(f/S)$ such that
			$Z(f-P_1):=\{x\in[a,b]:f(x)=P_1(x)\}$ has measure zero.}
		
		We claim that $\operatorname{sgn}(f-P_1)$ has at least $n$ sign changes in $[a,b]$.
		
		Indeed, suppose that it has only $m<n$ sign changes at points
		$x_1<\dots<x_m$. By the fundamental property of Haar spaces, there exists 
		$Q\in S \setminus\{0\}$ that changes sign exactly at
		$\{x_i\}_{i=1}^m$. Hence,
		\[
		\operatorname{sgn}(f(x)-P_1(x))Q(x)>0
		\quad\text{for every }x\in[a,b]\setminus\{x_1,\dots,x_m\}.
		\]
		Moreover, since $Z(f-P_1)$ has measure zero, we have
		\[
		0<
		\varphi_-\left(x,\lvert f(x)-P_1(x)\rvert\right)
		\le
		\varphi_+\left(x,\lvert f(x)-P_1(x)\rvert\right)
		\quad\text{for a.e. }x\in[a,b].
		\] and so
		\[\begin{split}
			0& <
			\varphi_-\left(x,\lvert f(x)-P_1(x)\rvert\right)\operatorname{sgn}(f(x)-P_1(x))Q(x) \\
			& \le
			\varphi_+\left(x,\lvert f(x)-P_1(x)\rvert\right) \operatorname{sgn}(f(x)-P_1(x))Q(x).
		\end{split},\]
		for a.e. $x\in[a,b]$. Applying the characterization inequality \eqref{eq:ineq_characterization} with
		$P=P_1$, we obtain
		\[
		\begin{aligned}
			0&< \int_{\{(f-P_1)Q>0\}}
			\varphi_-\left(x,\lvert f(x)-P_1(x)\rvert\right)
			\operatorname{sgn}(f(x)-P_1(x))Q(x)d\mu(x)
			\\
			&\quad+
			\int_{\{(f-P_1)Q<0\}}
			\varphi_+\left(x,\lvert f(x)-P_1(x)\rvert\right)
			\operatorname{sgn}(f(x)-P_1(x))Q(x)d\mu(x)
			\\
			&\le
			\int_{\{f=P_1\}}\varphi_+(x,0)|Q(x)|d\mu(x)=0.
		\end{aligned}
		\]
		a contradiction.
		
		Now, since $P_2$ is also a best approximant, the relation
		\[
		\left[f(x)-P_1(x)\right]\left[f(x)-P_2(x)\right]\ge0
		\]
		forces $f-P_2$ to vanish at each of the $n$ points where $f-P_1$ changes sign.
		Consequently, the continuous function
		\[
		P_1-P_2\in S
		\]
		has at least $n$ distinct zeros in $[a,b]$. Since $S$ is a Haar space of
		dimension $n$, every non-zero element of $S$ has at most $n-1$ zeros.
		Therefore,
		\[
		P_1-P_2\equiv0,
		\]
		which contradicts the assumption that $P_1$ and $P_2$ are distinct.
		
		\medskip\noindent\textit{Case 2: For every $P\in\mathcal{M}_\Phi(f/S)$ we have $\mu\left(Z(f-P)\right)>0$.}
		
		Let $P_1,P_2\in\mathcal{M}_\Phi(f/S)$ and set $P_\lambda = \lambda P_1 + (1-\lambda)P_2$ for $\lambda\in[0,1]$. By convexity, $P_\lambda\in\mathcal{M}_\Phi(f/S)$ for all $\lambda$. Since $\mu\left(Z(f-P_\lambda)\right)>0$ for every $\lambda$, by \cite[Lemma.~4--7, p.~109]{Rice64} there exist $\lambda_1\neq\lambda_2$ such that $Z(f-P_{\lambda_1})\cap Z(f-P_{\lambda_2})$ contains at least $n$ points. Hence $P_{\lambda_1}=P_{\lambda_2}$ on these $n$ points, which implies $P_{\lambda_1}=P_{\lambda_2}$ and therefore $P_1=P_2$.
	\end{proof}
	
	In order to extend the uniqueness results to multidimensional domains $\Omega \subset \mathbb{R}^d$, we bypass the Haar condition by employing the concept of $1$-spaces introduced in \cite{BenaventeCostaPonceFavier2026}, which generalizes the algebraic behavior of polynomials.

	\begin{definition}[$0$-space and $1$-space]
		Let $\Omega \subset \mathbb{R}^d$ be a compact set equipped with the Lebesgue measure $\mu$. A finite-dimensional subspace $S \subset C(\Omega)$ is called a \textbf{$0$-space} if for every non-zero $P \in S$,
		\[
		\mu\left(\{x \in \Omega : P(x) = 0\}\right) = 0.
		\]
		A $0$-space $S$ is called a \textbf{$1$-space} if there exists $h \in S$ such that $h(x) > 0$ for all $x \in \Omega$.
	\end{definition}
	
	\begin{lemma}\label{lem:interpolation_rewritten}
		Let \(\Phi\) be a Musielak--Orlicz function satisfying \(\Delta_2\) and let \(S\subset C(\Omega)\) be a \(1\)-space. If \(f\in C(\Omega)\) and \(P\in\mathcal M_\Phi(f/S)\), then there exists \(x_0\in\Omega\) such that \(f(x_0)=P(x_0)\).
	\end{lemma}
	
	\begin{proof}
		Assume $P(x) - f(x) > 0$ for all $x \in \Omega$. Since $S$ is a $1$-space, there exists $h \in S$ with $h(x) > 0$ on the compact set $\Omega$; hence $h$ is bounded away from zero. Choose
		\[
		0 < \varepsilon < \frac{\min_{x\in\Omega}\left(P(x)-f(x)\right)}{\max_{x\in\Omega}h(x)}.
		\]
		Then $P-\varepsilon h \in S$ and $P(x)-\varepsilon h(x) > f(x)$ for all $x$. By the strict monotonicity of $\Phi(x,\cdot)$,
		\[
		\Phi\left(x,|f(x)-(P(x)-\varepsilon h(x))|\right) < \Phi\left(x,|f(x)-P(x)|\right)
		\quad \text{for all} \quad x\in\Omega,
		\]
		and integrating over $\Omega$ yields $\rho_\Phi(f-(P-\varepsilon h)) < \rho_\Phi(f-P)$, contradicting the optimality of $P$. The case $P(x)-f(x)<0$ is analogous.
	\end{proof}
	
	\begin{theorem}\label{thm:uniqueness_local_convexity_rewritten}
		Let \(\Omega\subset\mathbb R^d\) be compact with \(\Omega=\overline{\operatorname{int}(\Omega)}\). Let \(\Phi\) be a Musielak--Orlicz function satisfying \(\Delta_2\). If  $\mu(\Omega \setminus \Omega_I) = 0$ and \(S\subset C(\Omega)\) is a \(1\)-space, then for every \(f\in C(\Omega)\) the set \(\mathcal M_\Phi(f/S)\) is a singleton.
	\end{theorem}
	
	\begin{proof}
		Let $P_1,P_2\in\mathcal{M}_\Phi(f/S)$ and assume $P_1\neq P_2$. By convexity, $P^*=\frac{P_1+P_2}{2}\in\mathcal{M}_\Phi(f/S)$. Since $S$ is a $0$-space, $\mu\left(\{x \in \Omega : P_1(x) = P_2(x)\}\right)=0$. By Lemma~\ref{lem:jackson_musielak},
		\[
		\left[f(x)-P_1(x)\right]\left[f(x)-P_2(x)\right]\ge 0 \quad\text{a.e. in }\Omega,
		\]
		which implies
		\begin{equation} \label{eq-2}
			|f(x)-P^*(x)| = \frac{1}{2}|f(x)-P_1(x)| + \frac{1}{2}|f(x)-P_2(x)|
			\quad\text{a.e. in }\Omega.
		\end{equation}
		
		By Lemma~\ref{lem:interpolation_rewritten}, there exists $x_0\in\Omega$ with $f(x_0)=P^*(x_0)$. Because $\Omega=\overline{\operatorname{int}(\Omega)}$, every neighbourhood of $x_0$ intersected with $\operatorname{int}(\Omega)$ has positive measure. By continuity of $f$ and $P^*$, for any $\delta>0$ the set
		\[
		U_\delta = \{x\in\Omega : |f(x)-P^*(x)|<\delta\}
		\]
		contains a relatively open neighbourhood of $x_0$ and hence $\mu(U_\delta)>0$.
		
		As $\mu(\Omega \setminus \Omega_I) = 0$, from  Proposition~\ref{prop:equiv_caso_I} there 
		exists an $\mathcal{A}$-measurable function $\beta:\Omega\to(0,\infty)$ such that 
		$\Phi(x, \cdot)$ is strictly convex on $(0, \beta(x))$ for a.e. $x \in \Omega$. Since $\beta(x)>0$ a.e., there exists $\delta_0>0$ such that the set
		\[
		E_{\delta_0} = \left\{x\in U_\delta : \beta(x)>\delta_0\right\}
		\]
		has positive measure. 
		
		Set $E = E_{\delta_0} \cap \{P_1\neq P_2\}$. Then $\mu(E)>0$, and for every $x\in E$ we have $|f(x)-P^*(x)|<\delta_0<\beta(x)$. By \eqref{eq-2} and the strict convexity of $\Phi(x,\cdot)$ on $[0,\beta(x)]$,
		\[
		\Phi\left(x,|f(x)-P^*(x)|\right)
		< \frac{1}{2}\Phi\left(x,|f(x)-P_1(x)|\right)
		+ \frac{1}{2}\Phi\left(x,|f(x)-P_2(x)|\right)
		\]
		for a.e.~$x\in E$. Integrating over $\Omega$ yields
		\[
		\rho_\Phi(f-P^*) < \frac{1}{2}\rho_\Phi(f-P_1)+\frac{1}{2}\rho_\Phi(f-P_2)
		= \inf_{Q\in S}\rho_\Phi(f-Q),
		\]
		which contradicts $P^*\in\mathcal{M}_\Phi(f/S)$.
	\end{proof}
	
	In \cite{BenaventeCostaPonceFavier2026}, a necessary and sufficient condition for uniqueness was established for Orlicz spaces when the generating function exhibits linear behavior near the origin and strict convexity elsewhere. We generalize this result to Musielak-Orlicz spaces over multidimensional domains.
	
	\begin{definition}
		For $g\in C(\Omega)$, the zero-set $Z(g)=\{x\in\Omega:g(x)=0\}$ is called a \textbf{$\gamma_g$-set} (with respect to $S$ and $\Phi$) if $0\in\mu_\Phi(g/S)$.
	\end{definition}
	
	\begin{remark}
		If $P\in\mathcal{M}_\Phi(f/S)$, then by linearity of $S$ we have $0\in\mu_\Phi((f-P)/S)$. Consequently $Z(f-P)$ is a $\gamma_{f-P}$-set.
	\end{remark}
	
	\begin{theorem}\label{thm:conditional_uniqueness_rewritten}
		Let \(\Omega\subset\mathbb R^d\) be compact with \(\Omega=\overline{\operatorname{int}(\Omega)}\). Let \(\Phi\) be a Musielak--Orlicz function satisfying \(\Delta_2\). Assume there exists an $\mathcal{A}$-measurable function $\beta:\Omega\to(0,\infty)$ with 
		$\operatorname{ess\,inf}_{x\in\Omega}\beta(x)>0$ such that for almost every 
		$x\in\Omega$ it holds
		\begin{enumerate}[label=(\roman*)]
			\item \(\Phi(x,\cdot)\) is affine on \([0,\beta(x)]\) and
			\item \(\Phi(x,\cdot)\) is strictly convex on \((\beta(x),\infty)\).
		\end{enumerate}
		Let \(S\subset C(\Omega)\) be a finite-dimensional \(0\)-space. Then, for every \(f\in C(\Omega)\), the set \(\mathcal M_\Phi(f/S)\) is a singleton if and only if for every \(f\in C(\Omega)\) and every \(P_1\in\mathcal M_\Phi(f/S)\) one of the following conditions holds
		\begin{enumerate}[label=(\alph*)]
			\item If \(|f(x)-P_1(x)|\le\beta(x)\) for almost every \(x\in\Omega\), then \(0\) is the only element of \(S\) that vanishes on the set \(\gamma_{f-P_1}\).
			\item The set \(\{x\in\Omega:|f(x)-P_1(x)|>\beta(x)\}\) has positive measure.
		\end{enumerate}
	\end{theorem}
	
	\begin{proof}
		\noindent\textit{Necessity.} Assume $\mathcal{M}_\Phi(f/S)=\{P_1\}$ and that (b) fails, so
		$$
		|f(x)-P_1(x)|\le\beta(x) \quad \text{a.e.}
		$$
		We must show that (a) holds. Suppose, for the sake of contradiction, that there exists $Q\in S\setminus\{0\}$ with $Q=0$ on the $\gamma_{f-P_1}$-set $Z(f-P_1)$.
		
		Since $\operatorname{ess\,inf}_{x\in\Omega}\beta(x)>0$, there exists $\lambda>0$ such that $\beta \ge \lambda$ a.e. on $\Omega$.
		Define the auxiliary continuous functions
		\[
		T(x)= \frac{\lambda Q(x)}{2\|Q\|}\quad \text{and} \quad h(x)=|T(x)|\operatorname{sgn}\left(f(x)-P_1(x)\right).
		\]
		Since $T$ vanishes on $Z(f-P_1)$, we obtain $h=0$ on $Z(f-P_1)$. Because $\Phi(x,t)=w(x)t$ for $t\in[0,\beta(x)]$, we have $0<\varphi_-(x,t)=\varphi_+(x,t)=w(t)$ for $t\in[0,\beta(x)]$. Thus, Theorem~\ref{thm:characterization} applied to $P_1$ for $\pm R \in S$ yields
		\begin{equation*}
			\left|\int_{\Omega\setminus Z(f-P_1)} w(x)\operatorname{sgn}\left(f(x)-P_1(x)\right)R(x)d\mu(x)\right| \le \int_{ Z(f-P_1)} \varphi_+(x,0)|R|d\mu
		\end{equation*}
		for all $R\in S$. So,
		\[
		\left|\int_{\Omega\setminus Z(h)} w(x)\operatorname{sgn}\left(h(x)\right)R(x)d\mu(x)\right| \le \int_{ Z(h)} \varphi_+(x,0)|R|d\mu,
		\text{ for all } R\in S,
		\]
		since $\operatorname{sgn}(h)=\operatorname{sgn}(f-P_1)$ a.e. and $Z(f-P_1) \subset Z(h)$. Hence $0\in\mu_\Phi(h/S)$ and therefore
		$Z(h)$ is a $\gamma_h$-set.
		
		In particular, taking $R=Q$ we get
		\[
		\int_{\Omega} w(x)\operatorname{sgn}(h(x))Q(x)d\mu(x)=0.
		\]
		Now set $P_\varepsilon:=\varepsilon T$ for $0<\varepsilon<1$. As
		$$
		h-\varepsilon T= |T|\left(\operatorname{sgn}(f-P_1)-\varepsilon \operatorname{sgn}(T)\right),
		$$
		it is easy to see that
		$$
		\operatorname{sgn}(h-\epsilon T)=\operatorname{sgn}(f-P_1), \quad  \text{a.e.}\quad \text{and} \quad |h-\epsilon T| \le 2\|T\|\le \lambda \le \beta(x).
		$$
		Therefore $|h(x)-\varepsilon T(x)|=\operatorname{sgn}(h(x))(h(x)-\varepsilon T(x))=|h(x)|-\varepsilon \operatorname{sgn}(h(x)) T(x)$ a.e. and, consequently,
		\begin{align*}
			\rho_\Phi(h-\varepsilon T)
			&=\int_\Omega w(x)|h(x)-\varepsilon T(x)|d\mu(x) \\
			&=\int_\Omega w(x)|h(x)|d\mu(x)
			-\varepsilon\int_\Omega w(x)\operatorname{sgn}(h(x))T(x)d\mu(x) \\
			&=\rho_\Phi(h).
		\end{align*}
		Thus $\varepsilon T\in\mu_\Phi(h/S)$ for all $0<\varepsilon<1$, and we reach a contradiction.
		
		\medskip \noindent \textit{Sufficiency.} Assume that for every $P\in\mathcal{M}_\Phi(f/S)$ 
		either (a) or (b) holds. Let $P_1,P_2\in\mathcal{M}_\Phi(f/S)$ with 
		$P_1\neq P_2$. By convexity of $\rho_\Phi$,  $P^*\in\mathcal{M}_\Phi(f/S)$. 
		Since $S$ is a $0$-space, $\mu\left(\{x \in \Omega : P_1(x) = P_2(x)\}\right)=0$, so $P_1\neq P_2$ 
		a.e. By Lemma~\ref{lem:jackson_musielak}, $g_i:=f-P_i$ satisfy 
		$g_1(x)g_2(x)\geq 0$ for a.e. $x\in\Omega$, so $g_1$ and $g_2$ have the 
		same sign or at least one vanishes a.e. In either case,
		\begin{equation} \label{eq-Id}
			|f(x)-P^*(x)| = \frac{1}{2}|f(x)-P_1(x)|+\frac{1}{2}|f(x)-P_2(x)|
			\quad\text{for a.e. }x\in\Omega. 
		\end{equation}
		
		\medskip\noindent\textit{Case (b).} Suppose $\mu(E)>0$ where 
		$E=\{x\in\Omega:|f(x)-P_1(x)|>\beta(x)\}$. 
		By Lemma~\ref{lem:jackson_musielak}, $|f(x)-P_1(x)|\neq|f(x)-P_2(x)|$ 
		for a.e. $x\in E$. Fix such an $x\in E$ and set $s=|f(x)-P_1(x)|>\beta(x)$ 
		and $t=|f(x)-P_2(x)|\geq 0$ with $s\neq t$. By~\eqref{eq-Id}, 
		$|f(x)-P^*(x)|=\frac{s+t}{2}$. Since $\Phi(x,\cdot)$ is affine with slope 
		$w(x)$ on $[0,\beta(x)]$,
		\begin{equation}\label{eq:affine}
			\Phi(x,u) = \Phi(x,0)+w(x)u \quad \text{for all } u\in[0,\beta(x)].
		\end{equation}
		Moreover, since $\Phi(x,\cdot)$ is strictly convex on $(\beta(x),\infty)$, 
		its right derivative satisfies $\varphi_+(x,u)>w(x)$ for all $u>\beta(x)$, 
		so that for $s>\beta(x)$,
		\begin{equation}\label{eq:strict_slope}
			\begin{split}
				\Phi(x,s)& =\Phi(x,\beta(x))+\int_{\beta(x)}^s\varphi_+(x,u)du 
				> \Phi(x,\beta(x))+w(x)(s-\beta(x)) \\ & =\Phi(x,0)+w(x)s.
			\end{split}
		\end{equation}
		We claim that
		\begin{equation}\label{eq:strict_conv_E}
			\Phi\left(x,\frac{s+t}{2}\right) 
			< \frac{1}{2}\Phi(x,s)+\frac{1}{2}\Phi(x,t),
		\end{equation}
		and consider two subcases according to the position of $\frac{s+t}{2}$ 
		relative to $\beta(x)$.
		
		\medskip\noindent\textit{Subcase $\frac{s+t}{2}>\beta(x)$.} Both 
		$\frac{s+t}{2}$ and $s$ lie in $(\beta(x),\infty)$. If 
		\eqref{eq:strict_conv_E} were an equality, then $\Phi(x,\cdot)$ would be 
		affine on $[t,s]$, and in particular on $[\beta(x),s]$, contradicting the 
		strict convexity of $\Phi(x,\cdot)$ on $(\beta(x),\infty)$. Hence 
		\eqref{eq:strict_conv_E} holds.
		
		\medskip\noindent\textit{Subcase $\frac{s+t}{2}\leq\beta(x)$.} Then 
		$t\leq\frac{s+t}{2}\leq\beta(x)<s$, so by \eqref{eq:affine},
		\begin{equation*}
			\begin{split}
				\Phi\left(x,\frac{s+t}{2}\right) 
				&= \Phi(x,0)+w(x)\frac{s+t}{2} \\
				&= \frac{1}{2}\left(\Phi(x,0)+w(x)t\right)+\frac{1}{2}\left(\Phi(x,0)+w(x)s\right) \\
				&= \frac{1}{2}\Phi(x,t)+\frac{1}{2}\left(\Phi(x,0)+w(x)s\right).
			\end{split}
		\end{equation*}
		By \eqref{eq:strict_slope}, $\Phi(x,s)>\Phi(x,0)+w(x)s$, and therefore
		\[
		\frac{1}{2}\Phi(x,s)+\frac{1}{2}\Phi(x,t) 
		> \frac{1}{2}\left(\Phi(x,0)+w(x)s\right)+\frac{1}{2}\Phi(x,t)
		= \Phi\left(x,\frac{s+t}{2}\right),
		\]
		which gives \eqref{eq:strict_conv_E}.
		
		\medskip In both subcases \eqref{eq:strict_conv_E} holds for a.e. $x\in E$. 
		On $\Omega\setminus E$, the standard convexity of $\Phi(x,\cdot)$ gives the 
		non-strict inequality. Since $\mu(E)>0$, integrating over $\Omega$ yields
		\[
		\rho_\Phi(f-P^*) 
		< \frac{1}{2}\rho_\Phi(f-P_1)+\frac{1}{2}\rho_\Phi(f-P_2)=\rho,
		\]
		contradicting $P^*\in\mathcal{M}_\Phi(f/S)$.
		
		\medskip\noindent\textit{Case (a).} Suppose (b) fails and (a) holds, so 
		$|f(x)-P_1(x)|\leq\beta(x)$ for a.e. $x\in\Omega$. Let $Q=P_1-P_2\in S$. 
		For a.e. $x\in Z(f-P^*)$, equation \eqref{eq-Id} gives 
		$|f(x)-P_1(x)|=|f(x)-P_2(x)|=0$ a.e., whence $Q(x)=P_1(x)-P_2(x)=0$ 
		for a.e. $x\in Z(f-P^*)$. Thus $Q$ vanishes a.e. on $Z(f-P^*)$, which is 
		a $\gamma_{f-P^*}$-set since $P^*\in\mathcal{M}_\Phi(f/S)$. By condition 
		(a) applied to $P^*$, we conclude $Q=0$, i.e. $P_1=P_2$, a contradiction.
	\end{proof}

	\section*{Declarations}
	
	\subsection*{Funding:} This work was supported by CONICET, Universidad Nacional de San Luis (Grant PI UNSL 03/3026), Universidad Nacional de R\'io Cuarto (Grant PPI 18/C614-2), and Universidad Nacional de La Pampa, Facultad de Ingenier\'ia (Grant Resol. Nro. 203/23).
	
	\subsection*{Conflict of interest:} The authors declare that they have no conflict of interest.
	
	\subsection*{Data availability:} No datasets were generated or analyzed during the development of this paper.
	
	\subsection*{Code availability:} Not applicable.
	
	\subsection*{Author contribution:} All authors contributed equally to the development of this paper.


\begin{thebibliography}{99}
		
		\bibitem{BenaventeCostaPonceFavier2026}
		A. Benavente, J. Costa Ponce, and S. Favier, \textit{On the uniqueness of best approximation in Orlicz spaces}, Constructive Mathematical Analysis, vol. 9, no. 2, pp. 62--71, 2026.
		\url{https://doi.org/10.33205/cma.1883206}
		
		\bibitem{BenaventeFavierLevis2017}
		A. Benavente, S. Favier, and F. E. Levis, \textit{Existence and characterization of best $\varphi$-approximations by linear subspaces}, Advances in Pure and Applied Mathematics, vol. 8, no. 3, pp. 209--217, 2017. \url{http://dx.doi.org/10.1515/apam-2015-0069}
		
		\bibitem{CostaPonce2024}
		J. Costa Ponce, \textit{Mejores Aproximantes en Espacios de Orlicz, Interpolaci{\'o}n}, Ph.D. thesis, Universidad Nacional de San Luis, Argentina, 2024.\url{https://drive.google.com/file/u/0/d/1vNVIpAxe_LCdMCtdaW6Wp9Fx_IVD6Nnl/view}
		
		\bibitem{HarjulehtoHasto2019}
		P. Harjulehto and P. H{\"a}st{\"o}, \textit{Orlicz Spaces and Generalized Orlicz Spaces}, Lecture Notes in Mathematics, vol. 2227, Springer, Cham, 2019. \url{https://doi.org/10.1007/978-3-030-15100-3}
		
		\bibitem{key:K}
		S. Karlin, W. Studden, Tchebycheff systems: With applications in analysis and statistics. Pure and Applied Mathematics, Wiley \& Sons, New York, 1966. ISBN-13: 978-0470458655.
		
		\bibitem{Kilmer1990}
		S. Kilmer, W. Koz{\l}owski, and G. Lewicki, \textit{Best approximants in modular function spaces}, Journal of Approximation Theory, vol. 63, no. 3, pp. 338--367, 1990. \url{https://doi.org/10.1016/0021-9045(90)90126-B}
		
		\bibitem{Kopaliani2021}
		T. Kopaliani, N. Samashvili, and S. Zviadadze, \textit{Extension of the best polynomial approximation operator in Musielak-Orlicz spaces}, Publicationes Mathematicae Debrecen, vol. 99, no. 1-2, pp. 101--115, 2021. \url{https://doi.org/10.5486/PMD.2021.8851}
		
		\bibitem{Kozlowski1988}
		W. M. Kozlowski, \textit{Modular Function Spaces}, Marcel Dekker, Inc., New York, 1988. ISBN-13: 978-0824780012.
		
		\bibitem{Krasnoselskii1961}
		M. A. Krasnoselskii and I. A. B. Rutickii, \textit{Convex Functions and Orlicz Space}, translated by L. F. Boron, P. Noordhoff, 1961, pp. 23--28, 60--66. ISBN-13: 978-0677202105.
		
		\bibitem{Landers1980}
		D. Landers and L. Rogge, \textit{Best approximants in $L_{\Phi}$-spaces}, Zeitschrift f{\"u}r Wahrscheinlichkeitstheorie und Verwandte Gebiete, vol. 51, no. 2, pp. 215--237, 1980. \url{https://doi.org/10.1007/BF00536190}
		
		\bibitem{MendezLang2019}
		O. Mendez and J. Lang, \textit{Analysis on Function Spaces of Musielak-Orlicz Type}, Chapman and Hall/CRC Monographs and Research Notes in Mathematics, CRC Press, Boca Raton, 2019.\url{https://doi.org/10.1201/9781498762618}
		
		\bibitem{Musielak1983_monograph}
		J. Musielak, \textit{Orlicz Spaces and Modular Spaces}, Lecture Notes in Mathematics, vol. 1034, Springer-Verlag, Berlin, 1983. \url{https://doi.org/10.1007/BFb0072210}
		
		\bibitem{MusielakOrlicz1959}
		J. Musielak and W. Orlicz, \textit{On modular spaces}, Studia Mathematica, vol. 18, no. 1, pp. 49--65, 1959. ISSN: 0039-3223.
		
		
		
		\bibitem{Folland1999}
		G. B. Folland, \textit{Real Analysis: Modern Techniques and Their Applications}, 2nd edition, John Wiley \& Sons, 1999, p. 62. ISBN-13: 978-0471317166.
		
		\bibitem{Rice64} 
		J. Rice, \textit{The Approximation of Functions}, Addison-Wesley Publishing Company, Inc., London, 1964. ISBN-13: 978-0201064308.
		
		\bibitem{key:Sc}
		L. Schumaker, Spline functions: Basic theory, Cambridge University Press, Cambridge, 2007. \url{https://doi.org/10.1017/CBO9780511618994}
		
	\end{thebibliography}
	\end{document}